\documentclass[12pt,notitlepage]{amsart}

\usepackage{tikz}
\usepackage[latin9]{inputenc}
\usepackage{fullpage}
\usepackage{array}
\usepackage{amsthm,amsmath,amssymb,amsfonts}
\usepackage{setspace}
\usepackage{graphicx}
\usepackage{cleveref}

\crefname{lemma}{Lemma}{Lemmas}
\crefname{claim}{Claim}{Claims}
\crefname{corollary}{Corollary}{Corollaries}
\crefname{theorem}{Theorem}{Theorems}
\crefname{fact}{Fact}{Facts}
\crefname{conjecture}{Conjecture}{Conjectures}
\crefname{proposition}{Proposition}{Proposition}

\newtheorem{conjecture}{Conjecture}
\newtheorem{theorem}[conjecture]{Theorem}
\newtheorem{corollary}[conjecture]{Corollary}
\newtheorem{proposition}[conjecture]{Proposition}
\newtheorem{lemma}[conjecture]{Lemma}

\newcommand{\R}{\mathbb{R}}
\newcommand{\Ps}{\mathcal{P}}
\newcommand{\one}{\mathbf{1}}
\newcommand{\Ls}{\mathcal{L}}
\newcommand{\Z}{\mathbb{Z}}

\newcommand{\ceil}[1]{\left\lceil #1 \right\rceil}

\newcommand{\eps}{\varepsilon}
\newcommand{\floor}[1]{\left\lfloor #1 \right\rfloor}

\newcommand{\qaq}{\quad\text{and}\quad}

\newcommand{\seq}[1]{\left\{ #1 \right\}}
\newcommand{\set}[2]{\left\{ #1 : #2 \right\}}
\newcommand{\size}[1]{\left| #1 \right|}

\renewcommand{\tilde}{\widetilde}

\newcommand{\imod}[1]{\,(\text{mod }#1)}

\title{Sets of Rich Lines in General Position}
\author{G.\ Amirkhanyan, A.\ Bush, E.\ Croot, \and C.\ Pryby}

\begin{document}

\begin{abstract}
The Szemer\'edi-Trotter theorem implies that the number of lines incident to at least $k > 1$ of $n$ points in $\mathbb{R}^2$ is $O(n^2/k^3 + n/k)$. J.\ Solymosi conjectured that if one requires the points to be in a grid formation and the lines to be in general position---no two parallel, no three meeting at a point---then one can get a much tighter bound. We prove a slight variant of his conjecture: for every $\varepsilon > 0$ there exists some $\delta > 0$ such that for sufficiently large values of $n$, every set of lines in general position, each intersecting an $n \times n$ grid of points in at least $n^{1-\delta}$ places, has size at most $n^\varepsilon$. This implies a conjecture of Gy.\ Elekes about the existence of a uniform statistical version of Freiman's theorem for linear functions with small image sets.
\end{abstract}

\maketitle 


\section{Introduction}

The Szemer\'edi-Trotter theorem \cite{sztrotter} states that, given a set of $n$ points and a set of $m$ lines in $\R^2$, the number of incidences (that is, pairs $(p,\ell)$ where $p \in \ell$) between these points and lines is $O(n^{2/3}m^{2/3} + m + n)$. This bound is sharp up to a constant coefficient. An equivalent form of the theorem states that, given $n$ points and an integer $k > 1$, there are $O(n^2/k^3 + n/k)$ \emph{$k$-rich} lines (that is, lines that each contain at least $k$ of the $n$ points in the set).

Our goal in this paper is to prove a variant of a conjecture of J.\ Solymosi \cite{elekes}:
\begin{theorem}
	\label{thm:soly}
	For every $\eps > 0$, there exists $0 < \delta < \eps$ such that for sufficiently large $n=n(\eps,\delta)$, the following holds:
	
	If $A \subseteq \R$ has size $n$, then every set of at least $n^\eps$ lines in $\R^2$, each of which intersects $A \times A$ in at least $n^{1-\delta}$ points, contains either two parallel lines or three lines with a common intersection point.
\end{theorem}

The Szemer\'{e}di-Trotter theorem gives a bound of $O(n^{1+3\delta})$ $n^{1-\delta}$-rich lines for an arbitrary set of $n^2$ points. As a consequence of Solymosi's conjecture, requiring a grid structure in the set of points and general position in the set of lines gives a significant improvement to the Szemer\'{e}di-Trotter bound.

We will first prove a weaker form of the conjecture. This form turns out to be sufficient to prove the form above.

\begin{theorem}
	\label{thm:weaksoly}
	For every $\eps > 0$, there exists $0 < \delta < \eps$ such that for sufficiently large $n=n(\eps,\delta)$, the following holds:
	
	If $A \subseteq \R$, $\size{A} = n$, then every set of at least $n^{1-\eps}$ lines in $\R^2$, each of which intersects $A \times A$ in at least $n^{1-\delta}$ points, contains either two parallel lines or $C = C(\eps) > 0$ lines with a common intersection point.
\end{theorem}

We begin by listing the major tools and basic results needed to prove \cref{thm:soly}. In Section 4 we introduce a key result, \cref{thm:4cases}, describing the behavior of sets of rich lines under self-composition, the main part of which we prove in Section 5. In Section 6, we describe a consequence of \cref{thm:4cases} which allows us to extract a set of lines in ``nearly'' general position from the composition of a set of rich lines with itself. In Section 7, we prove \cref{thm:weaksoly}, and in Section 8, we use the result of Section 6 to showing that \cref{thm:soly} follows from \cref{thm:weaksoly}. We conclude in Section 9 by describing an application of \cref{thm:soly}: it implies a uniform statistical version of Freiman's theorem, generalized to sets of (affine) linear functions.


\section{Major Tools}

In this section we list the theorems we use in this paper. First, we restate the Szemer\'edi-Trotter theorem (in the form we shall use it):

\begin{theorem}[Szemer\'edi-Trotter \cite{tao-vu}]
	\label{thm:ST}
	If $P$ is any finite set of points and $k \geq 2$, then the number of $k$-rich lines in $P$ is 
	\[ O\left( \max\left( \frac{\size{P}^2}{k^3}, \frac{\size{P}}{k} \right) \right). \]
\end{theorem}

Before listing the remaining tools we'll be using, let us define some notation. Given two subsets $A,B$ of an additive abelian group, define the \emph{sumset} by 
	\[ A+B := \set{a+b}{a \in A, b \in B}. \]
We further define the \emph{iterated sumset} by $1A := A$ and, for $k > 1$, 
	\[ kA := (k-1)A + A \]

\begin{theorem}[Pl\"unnecke-Ruzsa \cite{tao-vu}]
	\label{thm:pluruz}
	If $A$ is a subset of an additive group such that $\size{2A} \leq K\size{A}$, then $\size{nA} \leq K^n\size{A}$ for all $n \geq 1$.
\end{theorem}

If $A,B$ are subsets of an additive abelian group, define the \emph{additive energy} of $A$ and $B$ as 
	\[ E(A,B) = \#\{(a,a',b,b') \in A \times A \times B \times B : a+b=a'+b' \}. \]

\begin{theorem}[Balog-Szemer\'edi-Gowers \cite{tao-vu}]
	\label{thm:BSG}
	Let $A$ be a subset of an additive group. Given $c > 0$, there exist constants $C_1,C_2 > 0$ dependent only on $c$ such that, if $E(A,A) \geq K^{-c} \size{A}^3$, then there is a subset $A' \subseteq A$ such that $\size{A'} \geq K^{-C_1}\size{A}$ and $\size{2A'} \leq K^{C_2}\size{A'}$.
\end{theorem}

If $A,B$ are subsets of the group of nonzero reals under the operation of multiplication, then we call their sumset (as defined above) the \emph{product set}, denoted $A.B$, and we call their additive energy (as defined above) \emph{multiplicative energy} instead.

\begin{theorem}[Croot-Hart \cite{crohar}]
	\label{thm:crohar}
	For every $c > 0$ there exist $\beta > 0$ and $k \geq 1$ such that the following holds for all sufficiently large $N$:
	
	Let $B \subseteq \R$ have size $\size{B} = N$. If $\size{B.B} \leq N^{1+\beta}$, then $\size{kB} \geq N^c$.
\end{theorem}

Let $A$ be a subset of an additive abelian group, and let $A^k$ denote the $k$-fold cartesian product of $A$ with itself. If $S \subseteq A^k$, then define 
	\[ \Sigma(S) := \set{ a_1 + \cdots + a_k }{ (a_1,...,a_k) \in S }. \]

\begin{theorem}[Borenstein-Croot \cite{borcro2}]
	\label{thm:borcroBSG}
	For every $0 < \eps < 1/2$ and $c > 1$, there exists $\delta > 0$ such that the following holds for all $k,n$ sufficiently large:
	
	Let $A$ be a subset of an additive abelian group with $\size{A} = n$, and let $S \subseteq A^k$. If $\size{S} \geq \size{A}^{k-\delta}$ and $\size{\Sigma(S)} < \size{A}^c$, then there exists a subset $A' \subseteq A$ with $\size{A'} \geq \size{A}^{1-\eps}$ such that for all $\ell \geq 1$, $\size{\ell A'} \leq \size{A'}^{c(1+\eps\ell)}$.
\end{theorem}


\section{Preliminary Definitions and Facts}

	Let $A \subset \R$ be a finite subset with $\size{A} = n$. We call a line $\ell$ in $\R^2$ \emph{$k$-rich} if it intersects at least $k$ points in $A \times A$. A line can be at most $n$-rich; we will concern ourselves mainly with lines that are $n^{1-\delta}$-rich for some small positive $\delta$.
	
	If $\ell: y = \lambda x + b$ is a line in $\R^2$, then $\ell^{-1}: y = \frac{1}{\lambda}x - \frac{b}{\lambda}$ is the line such that $\ell \circ \ell^{-1} = \ell^{-1} \circ \ell$ is the identity function on $\R^2$ (i.e., the line $y=x$).
	
	If $L$ is a set of lines in $\R^2$, a subset $S \subseteq L$ is a \emph{star family} if all members of $S$ intersect at some common point. We say $L$ is in \emph{general position} if no two lines in $L$ are parallel and $L$ has no star families of size $3$. If $L$ has no two parallel lines and no star families of size greater than some constant $C \geq 2$, then we say $L$ is in near-general position.
	
	In the next two lemmas, we establish that many pairs of lines in a set of rich lines can be combined to obtain lines of slightly less richness.

\begin{lemma} 
	\label{lemma:easy1}
	Given sets $A_1,...,A_k \subseteq [n]$, each of size at least $n^{1-\delta}$, we must have at least $k^2n^{-2\delta}/2$ pairs of sets $(A_i,A_j)$ with $\size{A_i \cap A_j} \geq n^{1-2\delta}/2$.
\end{lemma}
\begin{proof}
Let $B = \{ (i,j) : \size{A_i \cap A_j} \geq \frac{n^{1-2\delta}}{2} \}$.  For a contradiction, suppose $\size{B} < \frac{1}{2} k^2 n^{-2\delta}$. Then 
\begin{multline*}
	\sum_{i,j} \size{A_i \cap A_j} = \sum_{(i,j) \in B} \size{A_i \cap A_j} + \sum_{(i,j) \in B^c} \size{A_i \cap A_j} < n \cdot \frac{1}{2}k^2 n^{-2\delta} + \\ \left(k^2 - \frac{1}{2}k^2 n^{-2\delta}\right)\frac{n^{1-2\delta}}{2} < k^2 n^{1-2\delta}. \end{multline*} 
However, letting $d(x) := \#\{i \in [k]: x \in A_i\}$, we have by Cauchy-Schwarz
\[ \sum_{i,j} \size{A_i \cap A_j} = \sum_{x=1}^n d(x)^2 \geq \left( n^{-1/2} \sum_{x \in [n]} d(x) \right)^2 = n^{-1} \left( \sum_i \size{A_i} \right)^2 \geq k^2n^{1-2\delta}. \]
\end{proof}

\begin{lemma}
	\label{lemma:easy1.5}
	Let $A \subset \R$, and let $L$ be a set of lines in $\R^2$ such that each line in $L$ is $n^{1-\delta}$-rich in $A \times A$. Then, for at least $\frac{1}{2} \size{L}^2 n^{-2\delta}$ pairs of lines $(\ell,\ell') \in L \times L$, $\ell^{-1} \circ \ell$ is $\frac{1}{2} n^{1-2\delta}$-rich in $A \times A$.
\end{lemma}
\begin{proof}
For each line $\ell: y = \lambda x + b$, let $X(\ell) = \set{x \in A}{ \lambda x + b \in A }$, and similarly let $Y(\ell) = \set{ y \in A }{  \lambda^{-1}(y-b) \in A }$. Observe that $Y(\ell) = X(\ell^{-1})$. Thus, for $(\ell,\ell') \in L \times L$, if $\size{Y(\ell) \cap Y(\ell')} \geq \frac{1}{2} n^{1-2\delta}$, then $\ell^{-1} \circ \ell'$ is $\frac{1}{2} n^{1-2\delta}$-rich in $A \times A$. Observe that each $Y(\ell)$ has size at least $n^{1-\delta}$. By Lemma \ref{lemma:easy1}, at least $\frac{1}{2} \size{L}^2 n^{-2\delta}$ pairs of sets $(Y(\ell),Y(\ell'))$ have intersection of size at least $\frac{1}{2} n^{1-2\delta}$.
\end{proof}

	We define the operation $*$ by $\ell_1 * \ell_2 = \ell_1^{-1} \circ \ell_2$. This formalizes the notion described earlier of combining rich lines in $L$ to form new rich lines (at the cost of a small amount of richness). Given two sets $L,L'$ of $n^{1-\delta}$-rich lines, we would like to consider the set of lines $\ell*\ell'$ which retain a large amount of richness in $A \times A$.
	\begin{equation}
		\label{eq:prerichlines}
		\set{ \ell * \ell' }{ \ell \in L, \ell' \in L', \size{\ell*\ell' \cap (A \times A)} \geq n^{1-2\delta}/2 }.
	\end{equation}

\begin{corollary}
	\label{cor:easy2}
	Given a set $L$ of lines which are $n^{1-\delta}$-rich in $A \times A$, there exist at least $\frac{1}{2} \size{L} n^{-2\delta}$ distinct lines of the form $\ell * \ell'$ which are $\frac{1}{2} n^{1-2\delta}$-rich in $A \times A$.
\end{corollary}
\begin{proof}
	There can be at most $\size{L}$ pairs which map to a given line in $L*L$, or else there exists $\ell_1 \in L$ and $\ell_2 \neq \ell_3$ such that $\ell_1^{-1} \circ \ell_2 = \ell_1^{-1} \circ \ell_3$, a contradiction. By \cref{lemma:easy1.5}, the result therefore follows.
\end{proof}

In addition to the richness of our new lines in $A \times A$, we will want to have control over the number of pairs $(\ell_1,\ell_2)$ which map to the same line under $*$. Let $\Ps(\ell)$ denote the set of pairs $(\ell_1,\ell_2)$ such that $\ell_1*\ell_2 = \ell$; if $X$ is a set of lines, let $\Ps(X)$ be the union of the sets $\Ps(\ell)$ over all $\ell \in X$. For each $0 \leq i \leq \ceil{\log_2\size{L}}$, let $L_i$ be the set of those lines $\ell$ in the set (\ref{eq:prerichlines}) such that 
	\[ 2^{i-1} < \size{\Ps(\ell)} \leq 2^i, \]
and let
	\[ N_i = \sum_{\ell \in L_i} \size{\Ps(\ell)}. \]
Then 
	\[ N_0 + N_1 + \cdots + N_K = \#\left\{(\ell,\ell'): \size{\ell*\ell' \cap (A \times A)} \geq \frac{1}{2} n^{1-2\delta}\right\} \geq \frac{1}{2}\size{L}^2n^{-2\delta} \]
by \cref{lemma:easy1.5}. By the pigeonhole principle, at least one $N_i$ satisfies 
	\[ N_i \geq \frac{\size{L}^2n^{-2\delta}}{2\log_2{\size{L}}}. \]
For the maximal such $i$, we define $L*L$ to be 
	\[ L*L := \left\{ \ell*\ell' : (\ell,\ell') \in \Ps(L_i) \right\}. \]

If $L$ is a set of $n^{1-\delta}$ rich lines, we recursively define the sequence of $j$-fold $*$ operations on $L$ as follows: take $L^{*2} := L * L$ and $L^{*j} := L^{*(j-1)}*L^{*(j-1)}$. We remark that the operation $*$ is not associative: for example, $(L * L) * (L * L)$ will not in general equal $((L * L) * L) * L$.

\section{Lines in Near-General Position}

The following theorem illustrates the behavior of a near-general position set of lines under the operation of $*$.

\begin{theorem}
	\label{thm:4cases}
	For all $0 < \eps < 1$, there exists $0 < \alpha_0 < \eps$ such that for all $0 < \alpha < \alpha_0$, there exists $0 < \delta_0 < \alpha$ such that for all $0 < \delta < \delta_0$ and for finite sets $A$ with $\size{A} = n$ sufficiently large, the following holds:
	
	If $L$ is a set of at least $n^\eps$ lines in near-general position (with star families bounded by some constant $C = C(\eps,\alpha) > 0$) which are $n^{1-\delta}$-rich in $A \times A$, then:
	\begin{enumerate}
		\item[(i)] If $L*L$ contains a family $P$ of parallel lines, then $\size{P} \leq 2\size{L*L}n^{2\delta}/\size{L}$.
		\item[(ii)] If $L*L$ contains a star family $S$, then $\size{S} \leq 2C\size{L*L}n^{2\delta}/\size{L}$.
		\item[(iii)] If $P_\lambda$ denotes the set of lines in $L*L$ with common slope $\lambda$, then $\size{P_\lambda} \geq n^\alpha$ for at most $n^\alpha$ numbers $\lambda$.
		\item[(iv)] If $S_p$ denotes the set of lines in $L*L$ with common meeting point $p$, then $\size{S_p} \geq n^\alpha$ for at most $n^\alpha$ points $p$.
	\end{enumerate}
\end{theorem}

Conditions (i), (ii), and (iv) will be shown in this paper. Condition (iii) is shown to hold in \cite{borcro1}.


The proofs of conditions (i) and (ii) are similar. Given a line $\ell \in L * L$, recall that $\Ps(\ell)$ denotes the set of pairs $(\ell_1,\ell_2) \in L \times L$ such that $\ell_1*\ell_2 = \ell$. If $X \subset L*L$, define $\Ps(X) := \bigcup_{\ell \in X} \Ps(\ell)$.

\subsection{Large Families of Parallel Lines}

\begin{proof}[Proof of Theorem \ref{thm:4cases}(i).]
	Suppose there is a family $P \subseteq L*L$ of parallel lines with 
	\[ \size{P} > \frac{2\size{L*L}}{\size{L}}n^{2\delta} \geq \frac{\size{L}}{2^i}, \]
where $i$ is the maximal index between $0$ and $\ceil{\log_2\size{L}}$ such that 
	\[ \#\set{ \ell \in L*L }{ 2^{i-1} < \size{\Ps(\ell)} \leq 2^i } \geq \frac{\size{L}^2 n^{-2\delta}}{2\log_2\size{L}}. \]
Then the total number of pairs mapping to lines in $P$ under $*$ will be 
	\[ \size{\Ps(P)} = \sum_{\ell \in P} \size{\Ps(\ell)} > \size{P} 2^i > \size{L}. \]
Since the lines of $L$ have distinct slopes, it follows that there exist two distinct pairs $(\lambda x+b,\lambda'x + b')$ and $(\lambda x+b,\lambda''x+b''')$ which each map to some line in $P$. But then $\lambda' = \lambda''$, and by the distinctness of the pairs it follows that $b' \neq b'''$. Thus, two lines in $L$ are parallel, contradicting the hypothesis that they are in near-general position.
\end{proof}

\subsection{Large Star Families}

\begin{proof}[Proof of Theorem \ref{thm:4cases}(ii).]
	It is sufficient to consider the case that the lines in $S$ intersect on the $y$-axis. If not, suppose the center of $S$ is $(x_0,y_0)$, and consider the grid $A' \times A'$, where $A'$ is the translate $A - x_0$. Suppose the pair $(\ell_1,\ell_2)$ of $n^{1-\delta}$-rich lines in $A \times A$ maps to a line in $S$, where $\ell_1: y = \lambda_1 x + b_1$ and $\ell_2: y = \lambda_2 x + b_2$. Then 
	\[ \ell_1 * \ell_2: y = \frac{\lambda_2}{\lambda_1} x + \frac{b_2-b_1}{\lambda_1} \]
contains the point $(x_0,y_0)$. Let $\ell_1',\ell_2'$ be the translates of $\ell_1,\ell_2$ down by $x_0$ and left by $x_0$; then 
	\[ \ell_1': y = \lambda_1 x + \lambda_1 x_0 + b_1 - x_0 \qaq \ell_2': y = \lambda_2 x + \lambda_2 x_0 + b_2 - x_0. \]
So $(\ell_1',\ell_2')$ maps to 
	\[ \ell_1' * \ell_2': y = \frac{\lambda_2}{\lambda_1} x + \frac{(\lambda_2-\lambda_1) x_0 + b_2 - b_1}{\lambda_1}. \]
At $x = 0$, we have 
	\[ y = \frac{\lambda_2}{\lambda_1}x_0 + \frac{b_2-b_1}{\lambda_1} - x_0 = y_0 - x_0. \]
That is to say, $(\ell_1',\ell_2')$ maps to a rich line passing through the point $(0,y_0-x_0)$. Thus, given a star family of lines $\frac{1}{2}n^{1-2\delta}$-rich in $A \times A$, we can construct a new $\frac{1}{2}n^{1-2\delta}$-rich star family of the same size in the translated grid $A' \times A'$ whose center lies on the $y$-axis.

	Now suppose there is a star family $S \subseteq L*L$ centered at $(0,y_0)$ with 
	\[ \size{S} > 2C\frac{\size{L*L}}{\size{L}}n^{2\delta} \geq C\frac{\size{L}}{2^i} \]
(where $i$ is taken as in the previous proof). Then the total number of preimages for lines in $S$ will be 
	\[ \size{\Ps(S)} = \sum_{\ell \in S} \size{\Ps(\ell)} > \size{S} 2^i > C\size{L}. \]
Since the lines of $L$ have distinct slopes, it follows that there exist $C+1$ distinct pairs in $L \times L$ mapping to $S$ such that the lines in the first coordinate of the pairs are the same:
	\[ (\lambda x+b,\lambda_1 x + b_1), (\lambda x+b,\lambda_1 x + b_2), ..., (\lambda x+b,\lambda_{C+1} x + b_{C+1}). \]
Since the $y$-intercepts of the output lines are the same, it follows that
	\[ \frac{b_1-b}{\lambda} = \frac{b_2-b}{\lambda} = \cdots = \frac{b_{C+1}-b}{\lambda} = y_0; \]
in other words, $b_1 = b_2 = \cdots = b_{C+1} = \lambda y_0 + b$. Thus, $L$ contains $C+1$ lines with a common $y$-intercept, contradicting the hypothesis that $L$ is in near-general position with star families bounded in size by $C$.
\end{proof}


\section{Star Families of Moderate Size}

Moving towards a proof of case (iv) of \cref{thm:4cases}, we begin with a technical combinatorial result.

\begin{lemma}
	\label{lemma:growth}
	For every $c > 0$ there exist $\alpha > 0$ and $k \geq 1$ such that the following holds for all sufficiently large $M$: suppose that $A_1,A_2,...,A_k \subseteq \R$ with $\size{A_i} = M$ 
 and $\size{A_i.A_i} \leq M^{1+\alpha}$ for all $i = 1,...,k$. If $S \subseteq A_1 \times A_2 \times \cdots \times A_k$ has size $\size{S} \geq M^{k-\alpha}$, then $\size{\Sigma(S)} \geq M^c$.
\end{lemma}

\begin{proof}
	Let $A_1,...,A_k,S$ be sets as in the statement of the lemma. Suppose that $\size{\Sigma(S)} < M^c$, where we assume that $M$ is large in terms of $k$, $c$, and $\alpha$---say, $M > 2^{2^{c+k+\alpha^{-1}+100}}$.

	Apply Theorem \ref{thm:borcroBSG}, taking $\alpha < \delta/2$, $\eps > 0$ sufficiently small (to be determined later), and $A = A_1 \cup \cdots \cup A_k$. Take $N = \size{A}$, observing that $M \leq N \leq k M$. Note that $S \subset A^k$, 
	\[ \size{S} \geq M^{k-\alpha} > (N/k)^{k-\alpha} > N^{k-\delta}, \quad\text{and}\quad \size{\Sigma(S)} < M^c \leq N^c. \]
Then there is a subset $A' \subseteq A$ with $\size{A'} \geq N^{1-\eps}$ such that for all $\ell \geq 1$, $\size{\ell A'} < N^{c(1+\eps\ell)}$.

	By the pigeonhole principle, $A'$ intersects some $A_j$, $1 \leq j \leq k$, in a set of size at least $\size{A'}/k > N^{1-\eps}/k$. Let $A''$ be that intersection, and note that this set satisfies the following:
	\[ \size{A''} > N^{1-\eps}/k, \quad \size{A''.A''} \leq \size{A_j.A_j} \leq M^{1+\alpha}, \quad\text{and}\quad \size{\ell A''} < N^{c(1+\eps\ell)}. \]
Expressing all of this in terms of $N' = \size{A''}$, we get 
	\[ \size{A''.A''} \leq M^{1+\alpha} \leq (k N')^{(1+\alpha)/(1-\eps)} \quad\text{and}\quad \size{\ell A''} < (k N')^{c(1+\eps\ell)/(1-\eps)}. \]
For $k \geq 1$, if $\ell$ is sufficiently large in terms of $c$, and if $\alpha$ and $\eps$ are sufficiently small in terms of $\ell$ and $k$, then these inequalities contradict Theorem \ref{thm:crohar}.
\end{proof}


From this lemma we prove a corollary which will have a direct application to the proof of condition (iv) of Theorem \ref{thm:4cases}.

\begin{lemma}
	\label{lem:graphs}
There is an absolute constant $c > 0$ such that for every $\eps > 0$ there exists $0 < \alpha_0 < \eps$ such that for all $0 < \alpha < \alpha_0$, there exists $0 < \delta_0 < \alpha$ such that the following holds for all $0 < \delta < \delta_0$ and sufficiently large $n = n(\eps,\alpha,\delta)$:

If $\seq{C_1,...,C_k}$ is a collection of sets of real numbers such that for all $i$, $\size{C_i} \geq n^\alpha$ and $\size{C_i.C_i} \leq \size{C_i}^{1+c_1\delta}$, $B$ is a set of real numbers with $\size{B} \geq n^\eps$, and $x_1,...,x_k \in \R$ are distinct constants such that for all $i$ and for each $\lambda \in C_i$, there are at least $\size{B}^{1-c\delta}$ pairs $(b,b') \in B \times B$ satisfying $\lambda(b - x_i) = b' - x_i$, then $k < n^{\alpha-c\delta}$.
\end{lemma}

	First we need another lemma:
	
\begin{lemma}
	\label{lemma:lemmaA}
	Suppose that $d_{i,j}$, $i = 1,...,k$ and $j = 1,...,N$ are real numbers satisfying $0 \leq d_{i,j} \leq L$. If $C \geq 0$ is defined by 
	\[ \sum_{i=1}^k \sum_{j=1}^N d_{i,j} = C L k N, \]
then there exists $i \in [k]$ such that for at least $kC^2/(2-C^2)$ indices $i' \in [k]$, we have 
	\begin{equation}
		\label{eq:lemmaA}
		\sum_{j=1}^N d_{i,j} d_{i',j} > \frac{1}{2} C^2 L^2 N.
	\end{equation}
\end{lemma}
\begin{proof}[Proof of Lemma.]
	By the Cauchy-Schwarz inequality, 
	\[ \sum_{1\leq i,i' \leq k} \sum_{j=1}^N d_{i,j}d_{i',j} = \sum_{j=1}^N \left( \sum_{i=1}^k d_{i,j} \right)^2 \geq \frac{1}{N} \left( \sum_{j=1}^N \sum_{i=1}^k d_{i,j} \right)^2 = C^2 L^2 k^2 N. \]
In particular, there must exist some $i \in [k]$ such that 
	\[ \sum_{i'=1}^k \sum_{j=1}^N d_{i,j}d_{i',j} \geq C^2 L^2 k N. \]
Fixing such an $i$, let $T$ denote the number of indices $i' \in [k]$ for which 
	\[ \sum_{j=1}^N d_{i,j}d_{i',j} \leq \frac{1}{2} C^2 L^2 N. \]
Then 
	\[ \frac{1}{2} T C^2 L^2 N + (k - T) L^2 N \geq C^2 L^2 k N, \]
so 
	\[ T \leq k \frac{1-C^2}{1-C^2/2}. \]
Thus, for at least 
	\[ k - T \geq \frac{k C^2}{2-C^2} \]
indices $i' \in [k]$, (\ref{eq:lemmaA}) holds.
\end{proof}

\begin{proof}[Proof of \cref{lem:graphs}.]
		By a dyadic pigeonhole argument, there exists a subcollection of the set $\{C_1,...,C_{k'}\}$ with size $\frac{k'}{\log_2(n)}$ and an integer $L \geq n^\alpha$ such that $L \leq \size{C_i} \leq 2L$ for each $C_i$ in the subcollection. Let $k$ be the number of elements in this subcollection, and reindex so that $C_1,...,C_k$ are the sets making up the subcollection.
	
	For each $i = 1,...,k$, construct the directed bipartite graph $G_i$ on vertex set $B_1 \sqcup B_2$ where $B_1 = B_2 = B$ and where $(b,b')$ is an edge if there exists $\lambda \in C_i$ such that $\lambda(b - x_i) = b' - x_i$. Letting $N = \size{B}$, the sum of the out-degrees in $B_1$ (and the sum of the in-degrees in $B_2$) is at least $LN^{1-O(\delta)}$ by our hypotheses on the size of $C_i$.
	
	If $G$ is a directed graph, define $\tilde{G}$ to be the graph obtained by reversing the orientation of each of $G$'s edges.
	
	If $G$ and $G'$ are two $(2^t+1)$-partite directed graphs whose vertex sets are $B_1 \sqcup \cdots \sqcup B_{2^t + 1}$, where $B_1 = \cdots = B_{2^t+1} = B$, then define the $(2^{t+1}+1)$-partite directed graph $G \wedge G'$ as follows: Let $V = B_1 \sqcup \cdots \sqcup B_{2^{t+1} + 1}$. For $m = 1,...,2^t$, let $(b_j,b_{j'}) \in B_m \times B_{m+1}$ be an edge in $G \wedge G'$ if and only if $(b_j,b_{j'}) \in B_m \times B_{m+1}$ is an edge in $G$. For $m = 2^t+1,...,2^{t+1}$, let $(b_j,b_{j'}) \in B_m \times B_{m+1}$ be an edge in $G \wedge G'$ if and only if $(b_j,b_{j'}) \in B_{m-2^t} \times B_{m+1-2^t}$ is an edge in $G'$.

Define $G_{i_1,i_2}$ to be $G_{i_1} \wedge \tilde{G}_{i_2}$. By \cref{lemma:lemmaA} (taking $C = N^{-O(\delta)}$) there is an index $1 \leq i_1 \leq k$ such that for at least $kN^{-O(\delta)}$ indices $1 \leq i_2 \leq k$, 
	\[ \sum_{j=1}^N d_{i_1,j} d_{i_2,j} \geq L^2 N^{1-O(\delta)}, \]
where $d_{i,j}$ is the number of directed edges in $G_i$ terminating at $b_j \in B_2$. This sum then counts the total number of paths of length $2$ in $G_{i_1,i_2}$, so the average number of paths in $G_{i_1,i_2}$ terminating at a particular $b \in B$ is at least $L^2 N^{-O(\delta)}$.

Now, fixing $i_1,...,i_{t-1}$, we can apply \cref{lemma:lemmaA} again to form a $(2^t + 1)$-partite graph 
	\[ G_{i_1,...,i_{t+1}} = G_{i_1,...,i_{t-1},i_t} \wedge \tilde{G}_{i_1,...,i_{t-1},i_{t+1}} \]
with $L^{2^t} N^{1-O(\delta)}$ length-$2^t$ paths corresponding to $2^t$-tuples $(\lambda_1,...,\lambda_{2^t})$ such that $\lambda_{4m+1} \in C_{i_1}$ and $\lambda_{4m} \in C_{i_1}^{-1}$ for all $m$.

Using the fact that 
	\[ \lambda_j(\beta_j - x_j) = \beta_{j+1} - x_j, \]
or, equivalently, 
	\[ \lambda_j \beta_j = \beta_{j+1} - x_j + \lambda_j x_j, \]
each of these $2^t$-tuples corresponds to an expression of the form:
	\begin{align*}
		\lambda_{2^t} \cdots \lambda_1 \beta_1 
		&= \lambda_{2^t} \cdots \lambda_2 (\beta_2 - x_1 + \lambda_1 x_1) \\
		&= \lambda_{2^t} \cdots \lambda_2 \beta_2 - \lambda_{2^t} \cdots \lambda_2 x_1 + \lambda_{2^t} \cdots \lambda_1 x_1 \\
		&\vdots \\
		&= (\beta_{2^t+1} - x_{2^t}) + \sum_{y=1}^{2^t} \left[ \prod_{j=y}^{2^t} \lambda_j \right] (x_y - x_{y-1}), 
	\end{align*}
where we define $x_0 = 0$.

By the pigeonhole principle, there exists a choice of $\beta_1$ and of the variables $\lambda_j$, $j \not\equiv 1 \imod{4}$, for which there are at least $L^{2^{t-2}} N^{-O(\delta)}$ paths in the $(2^t+1)$-partite graph starting at $\beta_1$ and utilizing the edges specified by the selected $\lambda_j$ (leaving at least $L^{2^{t-2}} N^{-O(\delta)}$ free choices of edges). Fixing such a $\beta_1$ and the variables $\lambda_j$ except $\lambda_{4s-3} \in C_i$ for all $1 \leq s \leq 2^{t-2}$, the left-hand side of the equality 
\begin{equation} \label{theequation1}
	\lambda_{2^t} \cdots \lambda_1 \beta_1 = (\beta_{2^t+1} - x_{2^t}) + \sum_{y=1}^{2^t} \left[ \prod_{j=y}^{2^t} \lambda_j \right] (x_y - x_{y-1}), 
\end{equation}
is an expression contained in the set 
	\[ C_i^{2^{t-2}} \cdot \beta_1 \cdot \prod_{\substack{ 1 \leq j \leq 2^t \\ j \not\equiv 1 \imod{4} }} \lambda_j, \]
which by \cref{thm:pluruz} has size at most $\size{C_i}^{1+O_t(\delta)}$.  Now we rewrite the right-hand-side, by grouping the terms indexed by $y$ into groups of four, starting at $y=4r+2$:  a typical such group will have sum
\begin{eqnarray} \label{theequation2}
&&\lambda_{4r+5} \lambda_{4r+6} \cdots \lambda_{2^t}\Bigl (\lambda_{4r+2} \lambda_{4r+3}\lambda_{4r+4}(x_{4r+2} - x_{4r+1}) \nonumber \\
&& \hskip0.5in + \lambda_{4r+3}\lambda_{4r+4} (x_{4r+3} - x_{4r+2}) \nonumber \\
&& \hskip0.5in + \lambda_{4r+4}(x_{4r+4} - x_{4r+3}) + x_{4r+5} - x_{4r+4}\Bigr ).
\end{eqnarray}
Conveniently, all the terms in the parentheses involve products of $\lambda_j$s, where $j \not \equiv 1 \pmod{4}$.  We can assume here that all of $x_{4r+1}, x_{4r+2}, x_{4r+3}, x_{4r+4}$ are distinct (the number of expressions resulting in duplications is small compared to the total number); and so, we have that $x_{4r+2} - x_{4r+1}$, $x_{4r+3} - x_{4r+2}$, $x_{4r+4} - x_{4r+3}$, and $x_{4r+5} - x_{4r+4}$ are all non-zero.  It follows, therefore, that for any choice of two of the parameters among $\lambda_{4r+2}, \lambda_{4r+3}, \lambda_{4r+4}$, there can be at most one possible choice of the remaining parameter that can make the expression (\ref{theequation2}) equal to $0$.  In fact, the number of paths through the graph resulting in a selection of the $\lambda_j$'s where at least one of the $2^{t-2}$ four-tuples equals $0$ is at most $L^{2^t-1}N^{1-O(\delta)}$.  But since there are many more paths than this, we can assume that there is a choice for the $\lambda_j$, $j \not \equiv 1 \pmod{4}$ where all the four-tuples are non-zero.  Re-expressing the right-hand-side of (\ref{theequation1}) in terms of these four-tuples, for this fixed choice of $\lambda_j$, $j \not \equiv 1 \pmod{4}$, we find that it is contained in the set 
	\[ \beta_{2^t+1} - x_{2^t} + \sum_{j=1}^{2^{t-2}} \kappa_j C_i^{4j}, \]
where $\kappa_j \neq 0$ are constants.  Furthermore, it turns out that for at least $|C_i|^{2^{t-2}-O(\delta)}$ vectors $(c_1,c_2,...,c_{2^{t-2}}) \in C_i \times C_i^2 \times \cdots \times C_i^{2^{t-2}}$, this expression is among the expressions in the right-hand-side of (\ref{theequation1}) that we can produce by Ruzsa-Plunnecke (since $|C_i.C_i| \approx |C_i|$).  Applying \cref{lemma:growth}, then, we quickly find that for $t$ large enough, the number of right-hand-side expressions exceeds $L^2$.  This contradicts the fact that the numberof left-hand side expressions is bounded by $L^{1+O_t(\delta)}$. This contradiction finishes the proof.
\end{proof}

Now, we may finally establish that if $L$ is in near-general position, then $L*L$ does not contain too many star families of ``moderate'' size.

\begin{proof}[Proof of \cref{thm:4cases}(iv).]
Suppose that there exist $k = n^\alpha$ star families $S_1,...,S_k \subseteq L*L$ with $\size{S_i} \geq n^\alpha$ for all $i$. Under this assumption, we will construct sets $B,C_1,...,C_k$ and distinct constants $x_1,...,x_k \in \R$ such that $\size{C_i} \geq n^\alpha$, such that $\size{C_i.C_i} \leq \size{C_i}^{1+O(\delta')}$ (for $\delta' \ll \alpha$) and such that for all $i$ and for all $\lambda \in C_i$, we have $\size{B}^{1-O(\delta')}$ pairs $(b,b') \in B \times B$ such that $\lambda(b-x_i) = b'-x_i$. This construction is forbidden by \cref{lem:graphs}, giving us a contradiction.

Begin by taking the $x_i$ to be the $x$-coordinates of the centers of the star families $S_1,...,S_k$ in $L*L$. Our first difficulty will be to show that the $x_i$ are distinct. Indeed, this may not be the case, for it is possible that many of the star families lie on common vertical lines. However, suppose that there are $K$ distinct vertical lines on which there are star families. Then there is some such line with at least $n^\alpha/K$ star families on it. Now, since a line $\ell: y = \lambda x + b$ is in $L*L$ if and only if its inverse $\ell^{-1}: x = \lambda y + b$ is also in $L*L$, it follows that there is a horizontal line with $n^\alpha/K$ star families on it, implying there are at least that many distinct vertical lines. Hence, $K \geq n^\alpha/K$, so $K \geq n^{\alpha/2}$. By choosing one star family from each vertical line and ignoring the rest (and reducing $\alpha$ to $\alpha/2$) we attain distinct $x$-coordinates for the centers of the star families.

Now, fix a star family $S_i$, and let $\Lambda_i$ be the set of slopes of the lines in $S_i$. Observe that lines in $S_i^{*2} := S_i * S_i$ will have slopes in the ratio set $Q_i := \Lambda_i/\Lambda_i$, lines in $S_i^{*3} := S_i^{*2} * S_i^{*2}$ will have slopes in $Q_i^2 = (\Lambda_i/\Lambda_i)^2$, and (in general) lines in $S_i^{*(j+2)}$, $j \geq 0$, will have slopes in the set $Q_i^{2^j}$. Now, not all elements of $Q_i^{2^j}$ will be slopes of lines in $S_i^{*(j+2)}$ (because some combined lines will not be rich enough in $A \times A$). Let $M_{i,j} \subset Q_i^{2^j}$ be the set of slopes of lines in $S_i^{*(j+2)}$.

Observe that $S_i * S_i$ is itself a star family centered at $(x_i,x_i)$, and $S_i * S_i$ contains the line $y = x$. Therefore, $S_i^{*j} \subseteq S_i^{*(j+1)}$ for all $j \geq 2$; hence $M_{i,j-1} \subseteq M_{i,j}$ for all $j \geq 1$. Moreover, since $y=x$ is in $S_i * S_i$, $M_{i,j}$ is closed under taking reciprocals for all $j$. Further note that lines in $S_i^{*j}$ will be $n^{1-2^{O(j)}\delta}$-rich in $A \times A$.

Let $\delta' > 0$ be a parameter such that $0 < \delta < \delta' < \alpha$. Suppose $\size{M_{i,j+1}} \geq \size{M_{i,j}}n^{\delta'\alpha}$ for all $j$ up to $m = \floor{2/\delta'\alpha}$.
Redefining $\delta$ if necessary, we may take $1 - 5^m \delta > 0$. For sufficiently large $n$, we then have
	\[ \size{M_{i,m+1}} \geq n^{\alpha+m\delta'\alpha} \geq n^2. \]
But, since each element of $M_{i,m+1}$ corresponds to a distinct rich line in $A \times A$, this violates \cref{thm:ST}. Therefore there exists some $j = j(i) < 2/\delta'\alpha$ such that 
	\[ \size{M_{i,j+1}} < \size{M_{i,j}}n^{\delta'\alpha} \leq \size{M_{i,j}}^{1+\delta'}. \]
Therefore, by \cref{lemma:easy1.5}, there are at least $\size{M_{i,j}}^{2-O(\delta')}$ pairs $(m,m') \in M_{i,j} \times M_{i,j}$ such that $m/m' \in M_{i,j+1}$. So the multiplicative energy of $M_{i,j}$ satisfies
	\[ E(M_{i,j},M_{i,j}) \geq \size{M_{i,j}}^{3-O(\delta')}. \]
By \cref{thm:BSG}, we conclude there is a subset $M_i' \subseteq M_{i,j}$ with small multiplicative doubling: $\size{M_i'.M_i'} \leq \size{M_i'}^{1+O(\delta')}$. Let $S_i'$ be those lines $\ell \in S_i^{*(j+2)}$ such that the slope of $\ell$ is in $M_i'$.

Now, let $\ell: \lambda x + (x_i - \lambda x_i) = \lambda(x - x_i) + x_i$ be a line in the stable star family $S_i'$. Since this line is rich in $A \times A$, there are $\size{A}n^{-O(\delta')}$ pairs $(a,a') \in A \times A$ such that $\lambda(a-x_i) = a' - x_i$. Taking $C_i = M_i'$ and $B = A$, this is the forbidden construction we desired above.
\end{proof}

\section{Constructing a Near-General Position Set of Lines}

Using Theorem \ref{thm:4cases}, we can extract a subset of lines in $L*L$ which is in near-general position: the subset will contain no two parallel lines, and all star families in the subset have size bounded by a constant $C = C(\eps,\alpha)$ independent of $n$.

\begin{corollary}
	\label{cor:almostgp}
	For all $0 < \eps < 1$ there exists $0 < \alpha_0 < \eps$ such that, for all $0 < \alpha < \alpha_0$, there exists $0 < \delta_0 < \alpha$ such that for all $0 < \delta < \delta_0$ and for sufficiently large $n$, the following holds:
	
	Let $A \subseteq \R$ be a finite set with $\size{A} = n$, and let $L$ be a set of at least $n^\eps$ lines which are all $n^{1-\delta}$-rich in $A \times A$. If $L$ contains no parallel lines and all star families in $L$ are bounded above in size by $C = C(\eps,\alpha)$, then there exists a subset $R \subseteq L*L$ such that
	\begin{itemize}
		\item $\size{R} \geq \size{L}n^{-c\alpha}$ for some absolute constant $c$, 
		\item $R$ contains no two lines which are parallel, and 
		\item at most $k = \ceil{\eps/\alpha}$ lines of $R$ pass through any given point of $\R^2$.
	\end{itemize}
\end{corollary}

We need a short lemma, which is easily proved by induction.
\begin{lemma}
	\label{lemma:binom}
	Let $k$ be a nonnegative integer and $0 < \gamma < 1$. Then 
	\[ \lim_{x \to \infty} x^{k\gamma} \cdot \frac{\binom{x}{x^{1-\gamma}-k}}{\binom{x}{x^{1-\gamma}}} = 1. \]
\end{lemma}

	\begin{proof}[Proof of \cref{cor:almostgp}.]
By Theorem \ref{thm:4cases}(iii), there are at most $n^\alpha$ families of parallel lines in $L*L$ with size greater than $n^\alpha$. By Theorem \ref{thm:4cases}(i), none of these families can have size greater than $2\size{L*L}n^{2\delta}/\size{L}$. Thus, deleting all of these lines from $L*L$ leaves us with a set of at least 
	\[ \size{L*L} - \frac{2\size{L*L}n^{\alpha+2\delta}}{\size{L}} > \frac{1}{2}\size{L*L} \]
lines.

The remaining families of parallel lines in this set have size at most $n^\alpha$, and these families are all disjoint. By picking a single representative from each family, we form a subset of $L*L$ of at least $\frac{1}{2}\size{L*L}n^{-\alpha}$ lines, no two of which are parallel. Invoking Theorem \ref{thm:4cases}(ii) and (iv), we remove from this subset all star families of size greater than $n^\alpha$ to leave us with a subset $L'$ with at least $\frac{1}{4}\size{L*L}n^{-\alpha}$ lines.

By Corollary \ref{cor:easy2}, there are at least $\size{L}n^{-3\delta}$ lines in $L*L$, so $L'$ contains at least $\size{L}n^{-2\alpha}$ lines.

Uniformly at random choose a subset $R \subseteq L'$ of $\ceil{\size{L'}n^{-c\alpha}}$ lines, where $c > 0$ is a parameter to be chosen later. The probability that a star family $S$ in $L'$ contains at least $k$ lines of $R$ is 
	\[ \frac{\binom{\size{S}}{k} \cdot \binom{\size{L'}-k}{\size{R}-k}}{\binom{\size{L'}}{\size{R}}} \leq \frac{n^{k\alpha}}{k!} \cdot  \frac{\binom{\size{L'}}{\size{R}-k}}{\binom{\size{L'}}{\size{R}}}. \]
Applying \cref{lemma:binom} with $x = \size{L'}$ and $x^{1-\gamma} = \size{R} = x^{1-c\alpha\log_x(n)}$, for large values of $n$ we have 
	\[ \frac{\binom{\size{L'}}{\size{R}-k}}{\binom{\size{L'}}{\size{R}}} = (1+o(1)) \cdot \size{L'}^{-kc\alpha \log_{\size{L'}}(n)} \leq 2n^{-kc\alpha}. \]
Since there are at most $n^{2\eps}$ star families, the expected number of star families with at least $k$ lines of $R$ is bounded by 
	\[ \frac{2}{k!} n^{2\eps+k(1-c)\alpha}. \]
Taking $k = \ceil{\frac{\eps}{\alpha}}$ and $c = 3$ makes this expected value less than $1$, meaning there is some choice of $R$ such that no star family has more than $\ceil{\eps/\alpha}$ lines in $R$.

Thus, $R$ is a near-general position subset of $L'$ (and therefore of $L$) with size at least $\size{L'}n^{-3\alpha} \geq \size{L}n^{-5\alpha}$.
\end{proof}

We remark that the proof still holds if $L*L$ above is replaced by $L'' \subseteq L*L$ so long as $\size{L''} \geq \size{L}n^{-c_0\alpha}$ for some $c_0 > 0$. We will use this modified version in the proof that \cref{thm:weaksoly} implies \cref{thm:soly}.


\section{Proof of Theorem 2}

Using \cref{cor:almostgp}, we are now ready to prove \cref{thm:weaksoly}. A major tool used will be the commutator graph, which we draw from \cite{elekes}.

Let $A \subseteq \R$, let $\delta > 0$, and let $L$ be a set of $n^{1-\delta}$-rich lines in $\R^2$. The \emph{commutator graph} on $L$ is the graph $G = (V,E)$, where 
	\[ V(G) = L*L \cup L^{-1}*L^{-1} \]
(with the minor change that we require minimum richness only $n^{1-5\delta}$ for each line in $L*L$ and $L^{-1}*L^{-1}$) and 
	\[ E(G) = \set{ \{f*g, g^{-1}*f^{-1} \} }{ f,g \in L, f*g \in L*L, g^{-1}*f^{-1} \in L^{-1}*L^{-1} }. \]
We draw attention to the fact that the lines $f*g$ and $g^{-1}*f^{-1}$ have the same slope. Hence, any edge of the commutator graph is between two parallel lines.

\begin{proof}[Proof of \cref{thm:weaksoly}]
Let $\eps > 0$, let $0 < \delta \ll \eps$, and let $A \subset \R$ with $n = \size{A} > 0$. Suppose for a contradiction that $L$ is a set of at least $n^{1-\eps}$ lines, all $n^{1-\delta}$-rich in $A \times A$, and that $L$ is in near-general position with star families bounded in size by a constant $C > 0$ independent of $n$. Consider the commutator graph on $L$.

If $\size{V(G)} \geq n^{1+4\delta}$, then we contradict \cref{thm:ST}, so let us assume that $\size{V(G)} < n^{1+4\delta}$. We claim that $\size{E(G)} \geq n^{2-6\delta}$. If this is true, then there is a vertex with degree at least $\size{E(G)}/\size{V(G)}$, so there is a connected component (corresponding to a set of parallel lines) of size $n^{1-10\delta}$, in contradiction with \cref{thm:4cases}(i). 

Let $S(f) = X(f) \times Y(f)$ for each $f \in L$. By applying \cref{lemma:easy1} to the collection of sets $S(f)$, where each set $S(f)$ has size at least $n^{2-2\delta}$, we must have at least $n^{2-4\delta}/2 \geq n^{2-5\delta}$ pairs $S(f),S(g)$ with $\size{S(f) \cap S(g)} \geq n^{2-4\delta}/2 \geq n^{2-5\delta}$. Note that for any sets $A_1,A_2,A_3,A_4$, $(A_1 \times A_3) \cap (A_2 \times A_4) = (A_1 \cap A_2) \times (A_3 \cap A_4)$. Thus, since $\size{S(f) \cap S(g)} \geq n^{2-5\delta}$, we have $\size{X(f) \cap X(g)} \geq n^{1-5\delta}$ and $\size{Y(f) \cap Y(g)} \geq n^{1-5\delta}$. Thus, we have at least $n^{2-5\delta}$ pairs $f,g \in L$ such that $f * g$ and $g^{-1} * f^{-1}$ are each $n^{1-5\delta}$-rich.

Let $f_i, g_i$ denote the lines such that $P_i := \{f_i*g_i, g_i^{-1}*f_i^{-1}\}$ is a pair of $n^{1-5\delta}$-rich lines. Given an index $i$, $f_i$ and $g_i$ intersect at a unique point $(x,y)$; it then follows that $y$ is the unique fixed point of $f_i*g_i$ and $x$ is the unique fixed point of $g_i^{-1}*f_i^{-1}$. Suppose there were $2C+2$ indices $i_1,...,i_{2C+2}$ such that $P_{i_j} = P_{i_k}$ for all $1 \leq j,k \leq 2C+2$. Then there would exist $C+1$ indices $i_{j_1},...,i_{j_{C+1}}$ such that 
	\[ f_{i_{j_1}}*g_{i_{j_1}} = \cdots = f_{i_{C+1}} * g_{i_{C+1}} \quad\text{and}\quad g_{i_{j_1}}^{-1} * f_{i_{j_1}}^{-1} = \cdots = g_{i_{C+1}}^{-1} * f_{i_{C+1}}^{-1}. \]
Since for each $1 \leq k \leq C+1$ there is a unique $(x,y)$ such that $f_{i_{j_k}}*g_{i_{j_k}}(y) = y$ and $g_{i_{j_k}}^{-1}*f_{i_{j_k}}^{-1}(x) = x$, it follows that $f_{i_{j_k}}$ and $g_{i_{j_k}}$ all intersect the point $(x,y)$. Since the $f_{i_{j_k}}*g_{i_{j_k}}$ must all have the same slope and $L$ has no parallel lines, we cannot have that $f_{i_{j_k}} = f_{i_{j_k'}}$ for $k \neq k'$ or else $g_{i_{j_k}} = g_{i_{j_k'}}$ as well, contradicting distinctness of the pairs. Similarly we must have $g_{i_{j_k}} \neq g_{i_{j_k'}}$ for $k \neq k'$. The collection  
	\[ \{f_{i_{j_k}} : 1 \leq k \leq C+1\} \cup \{g_{i_{j_k}} : 1 \leq k \leq C+1\} \]
must therefore contain at least $C+1$ distinct lines (a single line may appear as an $f_{i_j}$ at most once and as a $g_{i_j}$ at most once). But then we have a set of more than $C$ concurrent lines at $(x,y)$, contradicting the hypothesis that $L$ is in almost-general position.

Thus, for each edge $e$, there are at most $2C+2$ pairs $\{f_i \circ g_i^{-1}, g_i^{-1} \circ f_i\}$ equal to $e$, so the total number of edges in $G$ is at least $n^{2-5\delta}/(2C+2) \gg n^{2-6\delta}$, yielding a contradiction with \cref{thm:4cases}(i).
\end{proof}

We remark that taking $\delta < \eps/12$ is sufficient for the proof to go through.

\section{Theorem 2 Implies Theorem 1}

For $\ell \in L*L$, recall that $\Ps(\ell)$ is the set of all pairs $(f,g) \in L \times L$ such that $f*g = \ell$.

\begin{lemma}
	\label{lem:preimages}
	For all $0 < \eps < 1$, there exists $0 < \alpha_0 < \eps$ such that, for all $0 < \alpha < \alpha_0$, there exists $0 < \delta_0 < \alpha$ such that for all $0 < \delta < \delta_0$ and for sufficiently large $n$, the following holds:

Let $A \subseteq \R$ have size $n$, and let $L$ be a set of at least $n^\eps$ near-general position lines, all of which are $n^{1-\delta}$-rich in $A \times A$. Then there exists a set $L' \subseteq L*L$ such that $L'$ is a set of lines in near-general position, $\size{L'} > \size{L}n^{-5\alpha-4\delta}$, and for all $\ell \in L'$,
		\[ \size{\Ps(\ell)} \geq \frac{\size{L}^2}{2\size{L*L}n^{3\delta}}. \]
\end{lemma}
\begin{proof}
 Let 
 	\[ S := \{ (f,g) \in L \times L : f*g \text{ is $n^{1-5\delta}$-rich} \}. \]
By \cref{lemma:easy1.5}, $\size{S} \geq \size{L}^2 n^{-3\delta}$. Let 
	\[ T := \left\{ (f,g) \in S: \size{\Ps(f*g)} \leq \frac{\size{L}^2}{2\size{L*L}n^{3\delta}} \right\}. \]
If $\size{T} > \size{S}/2$, then we obtain an absurdity:
\begin{multline*}
	\size{L*L} = \sum_{(f,g)\in S} \frac{1}{\size{\Ps(f*g)}} = \sum_{(f,g) \in S \setminus T} \frac{1}{\size{\Ps(f*g)}} + \sum_{(f,g) \in T} \frac{1}{\size{\Ps(f*g)}} \geq \\ 
	\frac{1}{\size{L}} \size{S \setminus T} + \frac{2\size{L*L}n^{3\delta}}{\size{L}^2} \size{T} > \size{L*L}.
\end{multline*}
Thus, $\size{S\setminus T} \geq \size{L}^2 n^{-3\delta}/2 > \size{L}^2 n^{-4\delta}$. Letting $L' = \{ f*g : (f,g) \in S \setminus T \}$, we then have $\size{L'} \geq \size{L}n^{-4\delta}$. Apply \cref{cor:almostgp} to deduce that $L'$ contains a subset of $\size{L}n^{-5\alpha-4\delta}$ lines in near-general position.
\end{proof}
\begin{proposition}
	\cref{thm:weaksoly} implies \cref{thm:soly}.
\end{proposition}
\begin{proof}
	Let $L$ be a set of $n^\eps$ lines in general position, all of which are $n^{1-\delta}$-rich for some $\delta > 0$ to be chosen later. Fix $\alpha < \eps$, and suppose $\size{L^{*(k+1)}} \geq \size{L^{*k}}n^{5\alpha}$ for all $k$ up to $m = \floor{2/\alpha}$.
(By \cref{cor:almostgp}, we may further assume that $L^{*j}$ is in near-general position for all $j \leq k$ at the cost of a factor of $n^{4\alpha}$ each iteration.) Redefining $\delta$ if necessary, we can take $1 - 4 \cdot 5^m \delta > 0$. For sufficiently large $n$, we then have
	\[ \size{L^{*(m+1)}} \geq n^{\eps+m\alpha} \geq n^2.  \]
But this violates \cref{thm:ST}, so such an $m$ cannot exist. Therefore there exists $k < 2/\alpha$ such that 
	\[ \size{L^{*(k+1)}} < \size{L^{*k}}n^{5\alpha}. \]
In this case, let $L' = L^{*k}$ for the smallest such $k$ (such that the above inequality would now read $\size{L'*L'} < \size{L'}n^{5\alpha}$), let $\alpha' < 5\alpha$ such that $\alpha' \ll \eps$, let $N = \size{L'}$, and choose $\delta' \leq 5^k \delta$ such that $\delta' \ll \alpha'$.

By applying \cref{lem:preimages}, we can restrict our attention to a subset $L'' \subseteq L'*L'$ of size at least $N n^{-5\alpha'-4\delta'}$ such that all lines in $L''$ are in near-general position and, for all $\ell \in L''$, 
	\[ \size{\Ps(\ell)} \geq \frac{N^2}{2\size{L'*L'}n^{3\delta'}} \geq \frac{N}{2n^{5\alpha+3\delta'}} > \frac{N}{2n^{\alpha'+3\delta}}. \]
If $\ell$ is a line in $L''$, then $\ell = f*g$ for some $f,g \in L'$. We will then have at least 
	\[ \frac{1}{C} \size{L''} (N n^{-\alpha'-4\delta'})^2 \geq \frac{1}{C} N^3 n^{-5\alpha'-8\delta'} \gg N^3 n^{-6\alpha'} \]
solutions $(f,g,f',g') \in L \times L \times L \times L$ to the equation 
	\begin{equation}
		\label{eq:yint}
		f'*f(0) = g'*g(0)
	\end{equation}
(The factor of $1/C$ comes from the fact that $L''$ is a set of lines in near-general position, so at most $C$ lines will share a $y$-intercept.)

Now, fixing $f',g'$ in \eqref{eq:yint} and letting $f,g$ vary, we can interpret \eqref{eq:yint} as the line $f'*g'$, where the $x$ and $y$ variables are the $y$-intercepts of $f$ and $g$. Letting $B$ be the set of $y$-intercepts among lines in $L''$, we may interpret the above count of solutions to \eqref{eq:yint} as stating that many of the lines $f'*g'$ are $N n^{-7\alpha'}$-rich in the new grid $B \times B$. Indeed, let 
	\[ S := \{ (f',g') \in L' \times L': f'*g' \text{ is $N n^{-7\alpha'}$-rich in $B \times B$} \}, \]
and let $p(f',g')$ denote the number of points that $f'*g'$ intersects in $B \times B$. Then, for a contradiction, assume $\size{S} < N^2 n^{-8\alpha'}$. This implies the absurdity:
\begin{multline*}
	N^3 n^{-6\alpha'} = \sum_{(f',g') \in S} p(f',g') + \sum_{(f',g') \in S^c} p(f',g') < \size{S}\size{B} + \size{S^c}(N n^{-7\alpha'}) < \\
	\size{S} N n^{\alpha'} + (N^2 - \size{S})N n^{-7\alpha'} < \size{S} N n^{\alpha'}+N^3 n^{-7\alpha'} < 2N^3 n^{-7\alpha'} \ll N^3n^{-6\alpha'}.
\end{multline*}
(Note that this requires $N \gg n^{4\alpha'}$, which is satisfied provided $\alpha' \ll \eps$ because $N \gg n^\eps$.) Thus, $\size{S} \geq N^2n^{-8\alpha'}$, and that implies that we have at least $Nn^{-8\alpha'}$ lines that are all $N n^{-7\alpha'}$-rich in $B \times B$. Moreover, since $L'$ is in near-general position, we may extract a set of lines from $S*S$ that are in near-general position, and this set has size at least $N n^{-11\alpha'}$. However, this is in contradiction with \cref{thm:weaksoly}, since $S*S$ is a set of $N^{1-\gamma}$-rich lines in near-general position for some $\gamma > 0$, and $\size{S*S} \geq N^{1-12\gamma}$.
\end{proof}

\section{Connections to Freiman's Theorem}

\cref{thm:soly} implies a statistical version of a Freiman-type theorem on linear functions with small image sets. We primarily base the following exposition on selected sections of Elekes' survey paper, \cite{elekes}. We direct the reader to the references and proofs therein.

A \emph{proper generalized arithmetic progression}, or GAP, $G \subseteq \R$ is a set of the form 
	\[ G = \{ a_0 + r_1 a_1 + \cdots + r_d a_d : 0 \leq r_1 \leq n_1, ..., 0 \leq r_d \leq n_d \}, \]
where $a_0,...,a_d \in \R$, $n_1,...,n_d \in \Z^+$, and $\size{G} = n_1 \cdots n_d$.  An analogous definition may be made for a \emph{proper generalized geometric progression}, or GGP, $G \subseteq \R \setminus \{0\}$, where 
	\[ G = \{ a_0 \cdot r_1 a_1 \cdots r_d a_d : 0 \leq r_1 \leq n_1, ..., 0 \leq r_d \leq n_d \}, \]
where $a_0,...,a_d \in \R \setminus \{0\}$, $n_1,...,n_d \in \Z^+$, and $\size{G} = n_1 \cdots n_d$. In both these cases, the \emph{dimension} of $G$ is defined to be the minimal value of $d$ here, given $G$.

Recall Freiman's theorem on GAPs in the real numbers:
	\begin{theorem}[Freiman]
		\label{thm:freiman}
		Let $X,Y \subset \R$. If $\size{X}, \size{Y} \geq n$ and $\size{X+Y} \leq c n$, then $X \cup Y$ is contained in a GAP of dimension at most $d = d(c)$ and size at most $Cn$, $C = C(c)$ (that is, $d$ and $C$ do not depend on $n$).
	\end{theorem}
\noindent This theorem holds for GGPs in the reals as well.

If $X \subset \R$, let $E$ be a set of unordered pairs of elements in $X$. In this way, $E$ can be considered the edge set of a graph $G$ with vertex set $X$. We define 
	\[ X +_E X := \set{x'+x''}{\{x',x''\} \in E}. \]
	
	\begin{theorem}[Balog, Szemer\'edi]
		\label{thm:balog-sz}
		If $\size{E} \geq a\size{X}^2$ and $\size{X +_E X} \leq c\size{X}$, then there exists $\alpha = \alpha(a,c)$ such that some $\alpha\size{X}$ elements of $X$ are contained in a GAP of dimension $d = d(a,c)$ and size $C|X|$, $C = C(a,c)$.
	\end{theorem}
	
Elekes \cite{elekes} describes this version of the Balog-Szemer\'{e}di theorem as a ``statistical'' version of Freiman's theorem. He further observes that a ``uniform statistical'' hypothesis (that is, requiring a minimum degree proportional to $\size{X}$ in the graph $G$) guarantees that $X$ can be covered by a constant number of GAPs.

	\begin{theorem}[Elekes, Ruzsa]
		\label{thm:elek-ruz}
		Let $X \subset \R$ be finite, $\alpha > 0$ fixed, and $G$ a graph on $X$ with minimum degree $\alpha\size{X}$. Suppose that $\size{X +_E X} \leq c\size{X}$. Then, for all $\eps > 0$, $X$ can be partitioned into disjoint sets $X_1,...,X_k$ with the following properties:
		\begin{enumerate}
			\item Each $X_i$ is contained in a GAP $G_i$ of dimension $d$ and size $C\size{X}$ (the GAPs $G_i$ need not be disjoint);
			\item Each $X_i$ spans at least $\gamma\size{X}^2$ edges of $E$ (which implies that $k \leq 1/\gamma$); and 
			\item There are at most $\eps\size{X}^2$ ``leftover'' edges outside the $X_i$: 
				\[ \sum_{i<j} \sum_{x' \in X_i} \sum_{x'' \in X_j} \one_E(\{x',x''\}) \leq \eps\size{X}^2, \]
		\end{enumerate}
where $d = d(a,c,\eps)$, $C = C(\alpha,c,\eps)$, and $\gamma = \gamma(\alpha,c)$.
	\end{theorem}

	It turns out that Freiman's theorem extends to more general objects than sets of real numbers: sets of linear functions. Elekes observes that sets of lines with small composition sets satisfy a particular ``two extremities'' structure.
	
	Let $\Ls$ be the set of (affine) nonconstant linear functions $\R \to \R$.
\begin{theorem}[Elekes \cite{elekes}]
	\label{thm:lin-comp}
	For all $c,C > 0$, there exists $c' = c'(c,C) > 0$ with the following property.
	
	Let $\Phi,\Psi \subset \Ls$ be sets of $n$ lines each and $E \subset \Phi \times \Psi$ have size at least $c n^2$. Define 
	\[ \Phi \circ_E \Psi := \set{\phi \circ \psi}{(\phi,\psi) \in E}. \]
If $\size{\Phi \circ_E \Psi} \leq C n$, then there exist subsets $\Phi' \subseteq \Phi$ and $\Psi' \subseteq \Psi$ such that 
	\[ \size{(\Phi' \times \Psi') \cap E} \geq c' n^2 \]
and either 
	\begin{enumerate}
		\item[(i)] both $\Phi'$ and $\Psi'$ consist of functions whose graphs are all parallel (but the directions may be different for $\Phi'$ and $\Psi'$); or 
		\item[(ii)] both $\Phi'$ and $\Psi'$ consist of functions whose graphs all pass through a common point (which may be different for $\Phi'$ and $\Psi'$).
	\end{enumerate}
\end{theorem}

Using this theorem along with \cref{thm:freiman}, Elekes proves a result analogous to \cref{thm:freiman}: in fact, a true generalization of Freiman's theorem.

\begin{theorem}[Elekes \cite{elekes}]
	\label{thm:freiman-for-lines}
	For every $C > 0$ there are $C' = C'(C) > 0$, $C'' = C''(C) > 0$, and $d = d(C)$ with the following property.
	
	If $\Phi,\Psi \subset \Ls$ are sets of $n$ lines each and $\size{\Phi \circ \Psi} \leq C n$, then $\Phi^{-1} \cup \Psi$ is contained in a union of $C'$ parallel families or a union of $C'$ star families, and each of those families has size at most $C'' n$.
\end{theorem}

Elekes further notes that a uniform statistical version of this Freiman-type theorem for linear functions can be deduced from \cref{thm:elek-ruz}.
\begin{theorem}[Elekes \cite{elekes}]
	\label{thm:unif-freiman-for-lines}
	Let $\alpha > 0$ be fixed, $\Phi,\Psi$ as in \cref{thm:freiman-for-lines}, and $G(\Phi,\Psi,E)$ a bipartite graph with minimum degree at least $\alpha n$. If $\size{\Phi \circ_E \Psi} \leq C n$, then $\Phi^{-1} \cup \Psi$ is the union of a constant number of parallel and star families.
\end{theorem}

The previous theorems have focused on sets of lines $\Phi,\Psi$ whose composition set $\Phi \circ \Psi$ is small. Another direction we can explore is to study sets of lines $\Phi$ and points $H \subset \R$ whose \emph{image set} $\Phi(H)$ is small. In particular, \cref{thm:lin-comp} is equivalent to the following:
\begin{theorem}[Elekes \cite{elekes}]
	\label{thm:lin-image}
	For all $c,C > 0$, there exists $c' = c'(c,C) > 0$ with the following property.
	
	Let $H \subset \R$ and $\Phi \subset \Ls$ have size at least $n$, let $E \subset \Phi \times H$ have size at least $c n^2$, and define 
	\[ \Phi_E(H) := \set{\phi(h)}{(\phi,h) \in E}. \]
If $\size{\Phi_E(H)} \leq C n$, then there exists a parallel or star family $\Phi' \subseteq \Phi$ such that 
	\[ \size{E \cap (\Phi' \times H)} \geq c' n^2. \]
\end{theorem}

A result about small image sets analogous to \cref{thm:freiman-for-lines} holds: 
\begin{theorem}[Elekes \cite{elekes}]
	\label{thm:freiman-for-lines-image}
	For every $C > 0$ there are $C'=C'(C) > 0$ and $d = d(C) > 0$ with the following property.
	
	Let $\Phi \subset \Ls$ be a set of $n$ lines and $H \subset \R$ be a set of $n$ points. If $\size{\Phi(H)} \leq C n$, then $\Phi$ is contained in union of at most $C'$ parallel families or a union of at most $C'$ star families, and each of these families has size at most $C'' n$.
\end{theorem}

Elekes conjectured that a uniform statistical version of this theorem similar to \cref{thm:unif-freiman-for-lines} holds. We study one such statement in which there is a minimum degree requirement on only one side of the bipartite graph.
\begin{conjecture}[Elekes \cite{elekes}]
	\label{conj:unif-freiman-for-lines-image}
	Let $\alpha > 0$ be fixed, $\Phi,H$ as in \cref{thm:lin-image}, and $G(\Phi,H,E)$ a bipartite graph with degree at least $\alpha n$ for each $\phi \in \Phi$. If $\size{\Phi_E(H)} \leq C n$, then $\Phi$ is the union of a constant number of parallel and star families.
\end{conjecture}

In terms of cartesian products, this conjecture is equivalent to the following:
\begin{conjecture}[Elekes \cite{elekes}]
	\label{conj:unif-freiman-for-lines-cprod}
	If each of $c n$ lines is $c n$-rich in an $n \times n$ cartesian product, then the set of lines is the union of $C = C(c)$ parallel and star families.
\end{conjecture}

The following proof is not given explicitly in \cite{elekes} but can be inferred from similar arguments presented in that paper.

\begin{proof}[Proof of Equivalence.]
	Suppose \cref{conj:unif-freiman-for-lines-image} holds. Fix $0 < c < 1$, and let $\Phi$ be a set of $c n$ lines, each $c n$-rich in a $n \times n$ Cartesian product $A \times B$. Construct the bipartite graph $G(\Phi,A,E)$, where an edge connects $\phi \in \Phi$ and $a \in A$ whenever $\phi(a) \in B$. Then the degree of each $\phi \in \Phi$ is at least $c n$, so $\Phi$ is the union of a constant number of parallel and star families.

	For the other implication, suppose we have a set of $N$ lines $\Phi$, a set of $n$ points $H$, and an edge set $E \subseteq \Phi \times H$ such that $\size{\Phi_E(H)} \leq C n$ for $C > 0$. Observe that each line from $\Phi$ occurs in at least $\alpha n$ pairs of $E$, so we have $n \geq \alpha n$ lines which are each $\alpha n$-rich in the cartesian product $(H \cup \Phi_E(H)) \times (H \cup \Phi_E(H))$, which has size at most $(C+1)^2N^2$. So the lines are the union of a constant number of parallel and star families.
\end{proof}

\cref{conj:unif-freiman-for-lines-cprod}, and therefore \cref{conj:unif-freiman-for-lines-image}, would be implied by following conjecture of J\'{o}zsef Solymosi:
\begin{conjecture}[Solymosi \cite{elekes}]
	\label{conj:solymosi}
	Among the lines which are $c n$-rich in an $n \times n$ cartesian product, at most $C = C(c)$ can be in general position.
\end{conjecture}

\begin{proof}[Proof that \cref{conj:solymosi} implies \cref{conj:unif-freiman-for-lines-cprod}.] Given a set $L$ of $c n$ lines which are $c n$-rich in an $n \times n$ grid, let $L'$ be a maximum collection of these lines in general position. By \cref{conj:solymosi}, $\size{L'} = C$ for some constant $C$. Define $L(\lambda)$ to be the set of lines in $L$ with slope $\lambda$ and define $L(p)$ to be the set of lines in $L$ passing through a given point $p \in \R^2$. Then the union of $L(\lambda)$ over all $\lambda$ which are slopes of lines in $L'$ and $L(p)$ over all points $p$ which are intersections of pairs of lines in $L'$ must be $L$. So $L$ is the union of $C + C^2$ parallel and star families.
\end{proof}

The main result of the present paper, \cref{thm:soly}, yields an analogous result to \cref{conj:unif-freiman-for-lines-cprod}.

\begin{corollary}
	\label{cor:unif-freiman-for-lines-cprod-delta-eps}
	For every $\eps > 0$, there exists $0 < \delta_0 < \eps$ such that for all $0 < \delta < \delta_0$, the following holds for sufficiently large $n$:
	
	A set of lines that are each $n^{1-\delta}$-rich in an $n \times n$ cartesian product must be the union of $n^\eps$ parallel and star families.
\end{corollary}
\begin{proof} Given $\eps > 0$, choose $\delta > 0$ to satisfy \cref{thm:soly}. Then given a set $L$ of lines which are $n^{1-\delta}$-rich in an $n \times n$ grid, let $L'$ be a maximum collection of these lines in general position. So $\size{L'} \leq n^\eps$ for some constant $\eps > 0$. Define $L(\lambda)$ to be the set of lines in $L$ with slope $\lambda$ and define $L(p)$ to be the set of lines in $L$ passing through a given point $p \in \R^2$. Then the union of $L(\lambda)$ over all $\lambda$ which are slopes of lines in $L'$ and $L(p)$ over all points $p$ which are intersections of pairs of lines in $L'$ must be $L$. So $L$ is the union of $n^\eps + n^{2\eps}$ parallel and star families.
\end{proof}

Therefore, we also have the following uniform statistical Freiman-type theorem similar to \cref{conj:unif-freiman-for-lines-image}:
\begin{corollary}
	\label{cor:unif-freiman-for-lines-image-delta-eps}
	For every $\eps > 0$, there exists $0 < \delta_0 < \eps$ such that for all $0 < \delta < \delta_0$, the following holds for sufficiently large $n$:

	Let $\Phi,H$ be as in \cref{thm:lin-image} and $G(\Phi,H,E)$ be a bipartite graph with degree at least $n^{1-\delta}$ for each $\phi \in \Phi$. If $\size{\Phi_E(H)} \leq n^{1+\delta}$, then $\Phi$ is the union of $n^\eps$ parallel and star families.
\end{corollary}
\begin{proof}
	Let $A = H \cup \Phi_E(H)$. Each line in $\Phi$ is incident with at least $n^{1-\delta}$ edges of $E$, so each of the $n$ lines of $\Phi$ is $n^{1-\delta}$-rich in the $n^{1+\delta} \times n^{1+\delta}$ cartesian product $A \times A$. In other words, each line of $\Phi$ is $\size{A}^{1-\delta'}$-rich in $A \times A$, where $\delta' = 1 - \frac{1-\delta}{1+\delta}$. Since $\delta' \to 0$ as $\delta \to 0$, by choosing $\delta$ small enough, we can ensure through \cref{cor:unif-freiman-for-lines-cprod-delta-eps} that $\Phi$ is the union of $n^\eps$ parallel and star families.
\end{proof}

\end{document}